\documentclass[preprint,12pt]{elsarticle}
\usepackage{amsthm}
\usepackage{amssymb,amsmath}
\usepackage{pstricks}
\usepackage{ifthen}
\usepackage{calc}
\theoremstyle{plain}
\newtheorem{thm}{Theorem}
\newtheorem{lem}[thm]{Lemma}
\newtheorem{prop}[thm]{Proposition}
\newtheorem{cor}[thm]{Corollary}

\theoremstyle{definition}

\newproof{pf}{Proof}
\newproof{pol}{Proof of Lemma \ref{odd dicycle}}
\newproof{pop}{Proof of Proposition \ref{balanced diagonally simillar}}
 \makeatletter
\@addtoreset{thm}{chapter}
\@addtoreset{thm}{section}
\@addtoreset{thm}{subsection}

\makeatother

\begin{document}

\begin{frontmatter}

\title{Skew-signings of positive weighted digraphs }

\author[B]{Kawtar Attas}
\ead{kawtar.attas@gmail.com}

\author[B]{Abderrahim  Boussa\"{\i}ri\corref{cor1}}
\ead{aboussairi@hotmail.com}

\author[B]{Mohamed Zaidi}
\ead{zaidi.fsac@gmail.com}

\cortext[cor1]{Corresponding author}

\address[B]{Facult\'e des Sciences A\"{\i}n Chock, D\'epartement de
Math\'ematiques et Informatique,  Laboratoire de Topologie, Alg\`{e}bre, G\'{e}om\'{e}trie et Math\'{e}matiques discr\`{e}tes

Km 8 route d'El Jadida,
BP 5366 Maarif, Casablanca, Maroc}

\begin{abstract}
 An \emph{arc-weighted digraph }is a pair $(D,\omega)$ where $D$ is a digraph
and $\omega$ is an \emph{arc-weight function} that assigns\ to each arc $uv$
of $D$ a nonzero real number $\omega(uv)$. Given an arc-weighted digraph
$(D,\omega)$ with vertices $v_{1},\ldots,v_{n}$, the \emph{weighted adjacency
matrix} of $(D,\omega)$ is defined as the matrix $A(D,\omega)=[a_{ij}]$ where
$a_{ij}=\omega(v_{i}v_{j})$, if $v_{i}v_{j}\ $an arc of $D$ and $0\ $%
otherwise. Let $(D,\omega)$ be a  positive arc-weighted digraphs and assume that
$D$ is loopless and symmetric. A \emph{skew-signing} of $(D,\omega)$ is an
arc-weight function $\omega^{\prime}$ such that $\omega^{\prime}(uv)=\pm
\omega(uv)$ and $\omega^{\prime}(uv)\omega^{\prime}(vu)<0$ for every arc $uv$
of $D$. In this paper, we give necessary and sufficient conditions under which
the characteristic polynomial of $A(D,\omega^{\prime})$ is the same for every
skew-signing $\omega^{\prime}$ of $(D,\omega)$. Our Main Theorem generalizes a
result of Cavers et al (2012) about skew-adjacency matrices of graphs.
\end{abstract}


\end{frontmatter}


\section{Introduction}

 Let $G$ be a simple undirected and finite graph.
 An \emph{orientation} of $G$ is an assignment of a direction to
each edge of $G$ so that we obtain a directed graph $\overrightarrow{G}$. Let
$\overrightarrow{G}$ be an orientation of $G$. With respect to a labeling
$v_{1},\ldots,v_{n}$ of the vertices of $G$, the \emph{skew-adjacency} matrix of
$\overrightarrow{G}$ is the real skew-symmetric matrix $S(\overrightarrow
{G})=\left[  s_{ij}\right]  $, where $s_{ij}=1$ and $s_{ij}=-1$ if $v_{i}%
v_{j}$ is an arc of $\overrightarrow{G}$, otherwise $s_{ij}=s_{ji}=0$. The
\emph{skew-characteristic polynomial} of $\overrightarrow{G}$ is defined as
the characteristic polynomial of $S(\overrightarrow{G})$. This definition is
correct because skew-adjacency matrices of $\overrightarrow{G}$ with respect
to different labelings are permutationally similar and so have the same
characteristic polynomial. There are several recent works about
skew-characteristic polynomials of oriented graphs, one can see for example
\cite{Anuradha, Cavers, Cui, Gong, Shader, Yaoping}. Given a graph $G$, an open problem is to
find the number of possible orientations of  $G$ with distinct
skew-characteristic polynomials. In particular it is of interest to know
whether all orientations of $G$ can have the same skew-characteristic
polynomial. The following theorem, obtained by Cavers et al. \cite{Cavers}
gives an answer to this question.

\begin{thm}
\label{th cavers}The orientations of a graph $G$ all have the same
characteristic polynomial if and ond only if $G$ has no cycles of even length.
\end{thm}

A similar result was obtained by Liu and Zhang \cite{lui}. They proved that
all orientations of a graph $G$ have the same permanental polynomial if and
only if $G$ has no cycles of even length.

In this work, we will extend Theorem \ref{th cavers} to arc-weighted digraphs.
Recall that a \emph{directed graph} or a \emph{digraph} $D$ is a pair $D=(V,E)$ where $V$ is a set of
\emph{vertices} and $E$ is a set of ordered pairs of vertices called
\emph{arcs}. For $u,v\in$ $V$, an arc $a=(u,v)$ of $D$ is denoted by $uv$. An
arc of the form $uu$ is called a \emph{loop} of $D$. A \emph{loopless digraph}
is one containing no loops. A \emph{symmetric} digraph is a digraph such that if $uv$ is an arc then $vu$ is also an arc. A \emph{directed cycle} of length $t>0$ in a digraph $D$  is a
subdigraph  of $D$ with vertex set $\{v_{1},v_{2},\ldots,v_{t}\}$ and arcs $v_{1}%
v_{2},\ldots,v_{t-1}v_{t},v_{t}v_{1}$. Throughout the paper, we use the term "cycle" to refer to a "directed cycle" in a digraph.
 A cycle of length $t=2$ is called a \emph{digon}. A cycle is \emph{odd} ( resp. \emph{even}) if its length is
odd (resp. even).

An \emph{arc-weighted digraph} or more simply a \emph{weighted digraph} is a
pair $(D,\omega)$ where $D$ is a digraph and $\omega$ is a \emph{arc-weight
function} that assigns\ to each arc $uv$ of $D$ a nonzero real number $\omega(uv)$,
called the \emph{weight }of the arc $uv$.
Let $(D,\omega)$ be a weighted digraph with vertices $v_{1},\ldots,v_{n}$. The
\emph{weighted adjacency matrix} of $(D,\omega)$ is defined as the $n\times n$
matrix $A(D,\omega)=[a_{ij}]$ where $a_{ij}=\omega(v_{i}v_{j})$, if
$v_{i}v_{j}$ is an arc of $D$ and $0$ otherwise.

Every  $n\times n$ real matrix $M=\left[  m_{ij}\right] $ is the  weighted adjacency matrix of 
an unique weighted digraph $(D_{M},\omega_{M})$ with vertex set $\{1,...,n\}$. This digraph is called the \emph{weighted
digraph associated} to $M$ and defined as follows:  $ij$ is an arc of $D_{M}$
iff $m_{ij}\neq0$, and the weight of an arc $ij$ is $\omega_{_{M}}%
(ij)=m_{ij}$.

In the remainder of this paper, we consider only positive weighted loopless
and symmetric digraphs (which we abbreviate to \emph{pwls-digraphs}). Let
$(D,\omega)$ be a pwls-digraphs. A \emph{skew-signing} of $(D,\omega)$ is an
arc-weight function $\omega^{\prime}$ such that $\omega^{\prime}(uv)=\pm
\omega(uv)$ and $\omega^{\prime}(uv)\omega^{\prime}(vu)<0$ for every arc $uv$
of $D$.

The main result of this paper is the following theorem.

\begin{thm}
\label{principal} Let $(D,\omega)$ be pwls-digraph and let $\omega^{\prime}$
be a skew-signing of $\left(  D,\omega \right)  $. Then, the following
statements are equivalent:

\begin{description}
\item[i)] The characteristic polynomial of $(D,\omega^{\prime})$ is the same
for any skew-signing $\omega^{\prime}$ of $(D,\omega)$.

\item[ii)] $D$ has no even cycles of length more than $2$ and $A(D,\omega
)=\Delta^{-1}S\Delta$ where $S$ is a nonnegative symmetric matrix with zero
diagonal and $\Delta$ is a diagonal matrix with positive diagonal entries.
\end{description}
\end{thm}

Remark that a graph $G$ can be identified to the pwls-digraph obtained from
$G$ by replacing each edge $\left \{  u,v\right \}  $ by the two arcs $uv$ and
$vu$, both of them have weight $1$. Moreover, every orientation of $G$ can be
identified to a skew-signing of this digraph. Then, our main result is a
generalization of Theorem \ref{th cavers}.

\section{Cycle-symmetric digraphs}

We start with some formulas involving the characteristic polynomial of a
matrix and  its weighted associated digraph. For this, we need some notations
and definitions. Let $D$ be a digraph. A \emph{linear subdigraph} $L$ of $D$
is a vertex disjoint union of some cycles in $D$. A linear subdigraph $L$ of
$D$ is called \emph{even linear} if $L$ contains no odd cycle. Let
$\overrightarrow{\mathcal{L}}_{k}(D)$ (resp. $\overrightarrow{\mathcal{L}}%
_{k}^{e}(D)$) denote the set of all linear (resp. even linear) subdigraphs of
$D$  that cover precisely $k$ vertices of $D$. We usually
write this as $\overrightarrow{\mathcal{L}}_{k}$ (resp. $\overrightarrow
{\mathcal{L}}_{k}^{e}$) when no ambiguity can arise.

Let $A$ be a real matrix and let $(D,\omega)$ the weighted digraph  associated to $A$. We denote by $p_{A}(x)=\det(xI-A)=x^{n}+a_{1}%
x^{n-1}+\cdots+a_{n-1}x+a_{n}$ the characteristic polynomial of $A$.

 By using the classical definition of the determinant, we obtain the following formula:%

\begin{equation}
a_{k}=\underset{\overrightarrow{L}\in \overrightarrow{\mathcal{L}}_{k}}{\sum
}(-1)^{\left \vert \overrightarrow{L}\right \vert }\omega(\overrightarrow{L})
\label{2}%
\end{equation}
\  \  \  \  \  \  \  \  \  \  \  \  \  \  \  \  \  \

where $\left \vert \overrightarrow{L}\right \vert $ denotes the number of
cycles in $\overrightarrow{L}$ and $\omega(\overrightarrow{L})$ is the
product of all the weights of the arcs of $\overrightarrow{L}$.

In particular
\begin{equation}
\det(A)=(-1)^{n}a_{n}=(-1)^{n}\underset{\overrightarrow{L}\in \overrightarrow
{\mathcal{L}}_{n}}{\sum}(-1)^{\left \vert \overrightarrow{L}\right \vert }%
\omega(\overrightarrow{L}). \label{3}%
\end{equation}

If $A$ is skew-symmetric, then
\begin{equation}
a_{k}=\left \{
\begin{array}
[c]{c}%
0\text{ if }k\text{ is odd }\\
\underset{\overrightarrow{L}\in \overrightarrow{\mathcal{L}}_{k}^{e}}{\sum
}(-1)^{\left \vert \overrightarrow{L}\right \vert }\omega(\overrightarrow
{L})\text{ if }k\text{ is even }%
\end{array}
\right.  \label{4}%
\end{equation}

 We introduce now a special class of weighted symmetric digraphs called cycle-symmetric
digraphs. The characterization of these digraphs will be used in the proof of
our main theorem.

Given a symmetric digraph $D=(V,E)$ and a subdigraph $H=(W,F)$ of $D$, we
denote by $H^{\ast}$ the subdigraph of $D$ whose vertex set is $W$ and arc set is
$\left \{  vu:uv\in F\right \}  $. Consider now a positive arc-weight function $\omega$
of $D$ and let $q$ be a positive integer. We say that $(D,\omega)$ is $(\leq
q)$-\emph{cycle-symmetric} if $\omega(\overrightarrow{C})=\omega(\overrightarrow
{C}^{\ast})$ for every cycle $\overrightarrow{C}$ of $D$ of length at most
$q$. If $q=n$ then $(D,\omega)$ is said to be \emph{cycle-symmetric}.

We borrowed the terminology "cycle-symmetric" from  Shih and  Weng  \cite{Shih}. Following  this paper,
an $n\times n$ real matrix $\left[  a_{ij}\right]  $ is called
\emph{cycle-symmetric }  if the two conditions hold:

\begin{description}
\item[C1)] for $i\neq j\in \left \{  1,\ldots,n\right \}  $, $a_{ij}a_{ji}>0$ or
$a_{ij}=a_{ji}=0$;

\item[C2)] For any sequence of distinct integers $i_{1},\ldots,i_{k}$ from the set
$\left \{  1,\ldots,n\right \}  $, we have
\[
a_{_{i_{1}i_{2}}}a_{_{i_{2}i_{3}}}\cdots a_{_{i_{k-1}i_{k}}}a_{_{i_{k}i_{1}}%
}=a_{_{i_{1}i_{k}}}a_{_{i_{k}i_{k-1}}}\cdots a_{_{i_{3}i_{2}}}a_{_{i_{2}i_{1}%
}}%
\]

\end{description}

Obviously, a pwls-digraph  $(D,\omega)$ is always $(\leq
2)$-cycle-symmetric. Moreover, a pwls-digraph is cycle-symmetric if and only if its weighted adjacency matrix is cycle-symmetric.

The following theorem gives a characterization of cycle-symmetric matrices.
\begin{thm}
An $n\times n$ real matrix $A$ is cycle-symmetric if and only if there exists an invertible diagonal matrix $D$ such that $D^{-1}AD$ is symmetric.
\end{thm}

Different proofs of this theorem are given in \cite{maybee, Parter, Shih}.
As a consequence, we obtain the following characterization of cycle-symmetric pwls-digraphs.
\begin{cor}
\label{balanced diagonally simillar} Let $(D,\omega)$ be a pwls-digraph. Then,
the following statements are equivalent:
\begin{description}
\item[i)] $(D,\omega)$ is cycle-symmetric. 

\item[ii)] $A(D,\omega)=\Delta^{-1}S\Delta$ where $S$ is a nonnegative symmetric matrix with zero
diagonal and $\Delta$ is a diagonal matrix with positive diagonal entries.
\end{description}
\end{cor}

\section{Skew-signings of cycle-symmetric digraphs}

In this section, we study cycle-symmetric pwsl-digraphs $(D,\omega)$ such that the
characteristic polynomial of $(D,\omega^{\prime})$ is the same for any
skew-signing $\omega^{\prime}$ of $(D,\omega)$. More precisely, we will prove
the following Proposition.

\begin{prop}
\label{skew balanced}Let $(D,\omega)$ be a cycle-symmetric pwls-digraph. Then, the
following statements are equivalent:

\begin{description}
\item[i)] The characteristic polynomial of $(D,\omega^{\prime})$ is the same
for any skew-signing $\omega^{\prime}$ of $(D,\omega)$;

\item[ii)] $D$ contains no even cycle of length greater than $3$.
\end{description}
\end{prop}

Let $(D,\omega)$ is an arbitrary pwls-digraph and let $\omega^{\prime}$ be a
skew-signing $\omega^{\prime}$ of $(D,\omega)$. We consider the two arc-weight
functions defined as follows:  $\overline{\omega}(uv)=\sqrt{\omega
(uv)\omega(vu)}$ and $\widehat{\omega^{\prime}}(uv)=\frac{\omega^{\prime}%
(uv)}{\omega(uv)}\sqrt{\omega(uv)\omega(vu)}$ for every arc $uv$ of $D$.  We
have the following properties:

\begin{description}
\item[P1] The weighted adjacency matrix of  $(D,\overline{\omega})$ is 
symmetric.

\item[P2] $\widehat{\omega^{\prime}}$ is a skew-signing of $(D,\overline
{\omega})$ and $A(D,\widehat{\omega^{\prime}})$ is a skew-symmetric matrix.

\item[P3] If $(D,\omega)$ is a $(\leq q)$-cycle-symmetric digraph, then for every
dicycle $\overrightarrow{C}$ of $D$ with length $\leq$ $q$ we have
$\omega(\overrightarrow{C})=\overline{\omega}(\overrightarrow{C})$ and $\omega^{\prime}(\overrightarrow{C})=\widehat{\omega^{\prime}}(\overrightarrow
{C})$.
\end{description}

 We denote by $p_{(D,\omega)}(x):=x^{n}+a_{1}x^{n-1}%
+\cdots+a_{n-1}x+a_{n}$ the charateristic\ polynomial of $(D,\omega)$. The
charateristic\ polynomials of $(D,\omega^{\prime})$ and $(D,\widehat{\omega^{\prime}})$ are respectively denoted by
$p_{(D,\omega^{\prime})}(x):=x^{n}+b_{1}x^{n-1}+\cdots+b_{n-1}x+b_{n}$ and
$p_{(D,\widehat{\omega^{\prime}})}(x):=x^{n}%
+c_{1}x^{n-1}+\cdots+c_{n-1}x+c_{n}$.

From Formula (\ref{2}), we have $b_{1}=0$ and $b_{2}=-a_{2}$. In particular,
$b_{1}$ and $b_{2}$ are independent of $\omega^{\prime}$.

\begin{lem}
\label{coeff a balanced}If $(D,\omega)$ is $(\leq q)$-cycle-symmetric, then:%
\[%
\begin{array}
[c]{cc}%
b_{k}& =\left \{
\begin{array}
[c]{cc}%
0 & \text{if }k\text{ is odd}\\
\underset{\overrightarrow{L}\in \overrightarrow{\mathcal{L}}_{k}^{e}}{\sum
}(-1)^{\left \vert \overrightarrow{L}\right \vert }\omega^{\prime}%
(\overrightarrow{L}) & \text{if }k\text{ is even}%
\end{array}
\right.
\end{array}
\]
for $k=$ $1,\ldots,q$.
\end{lem}

\begin{proof}
Let $k\in$ $\left \{  1,\ldots,q\right \}  $. It follows from \textbf{P3} that
$\omega^{\prime}(\overrightarrow{L})=\widehat{\omega^{\prime}}(\overrightarrow
{L})$ for every $\overrightarrow{L}\in \overrightarrow{\mathcal{L}}_{k}$ and
hence by using Formula (\ref{2}), we have, $b_{k}=c_{k}$. Moreover, $A(D,\widehat
{\omega^{\prime}})$ is a skew-symmetric matrix then by formula (\ref{4}):
\[%
\begin{array}
[c]{cc}%
c_{k} & =\left \{
\begin{array}
[c]{cc}%
0 & \text{if }k\text{ is odd}\\
\underset{\overrightarrow{L}\in \overrightarrow{\mathcal{L}}_{k}^{e}}{\sum
}(-1)^{\left \vert \overrightarrow{L}\right \vert }\widehat{\omega^{\prime}%
}(\overrightarrow{L}) & \text{if }k\text{ is even}%
\end{array}
\right.
\end{array}
\]

Now, by applying again \textbf{P2}, we obtain

$%
\begin{array}
[c]{cc}%
b_{k}=c_{k} & =\left \{
\begin{array}
[c]{cc}%
0 & \text{if }k\text{ is odd}\\
\underset{\overrightarrow{L}\in \overrightarrow{\mathcal{L}}_{k}^{e}}{\sum
}(-1)^{\left \vert \overrightarrow{L}\right \vert }\omega^{\prime}%
(\overrightarrow{L}) & \text{if }k\text{ is even}%
\end{array}
\right.
\end{array}
$
\end{proof}

We denote by $\overrightarrow{\mathcal{C}}_{k}$ the set of cycles of length
$k$ of $D$. For a skew-signing $\omega^{\prime}$, this set is partitioned
into two subsets: $\overrightarrow{\mathcal{C}}_{k,\omega^{\prime}}^{+}$ and
$\overrightarrow{\mathcal{C}}_{k,\omega^{\prime}}^{-}$ where $\overrightarrow
{\mathcal{C}}_{k,\omega^{\prime}}^{+}$ (resp. $\overrightarrow{\mathcal{C}%
}_{k,\omega^{\prime}}^{-}$) is the set of cycles $\overrightarrow{C}$ such
that $\omega^{\prime}(\overrightarrow{C})>0$ (resp. $\omega^{\prime
}(\overrightarrow{C})<0$). For $k$ even, let $\overrightarrow{\mathcal{D}}_{k}$ denote the set
of all collections $\overrightarrow{L}$ of vertex disjoint digons that cover
precisely $k$ vertices in $D$.

\begin{cor}
\label{coeff}
\label{coeff}Let $q\geq4$. Assume that $(D,\omega)$ is $(\leq q-1)$-cycle-symmetric
and contains no even cycles of length $k\in$ $\left \{  3,\ldots,q-1\right \}
$ then
\[
b_{k}=\left \{
\begin{array}
[c]{ll}%
\underset{\overrightarrow{L}\in \overrightarrow{\mathcal{D}}_{k}}{\sum}%
\omega(\overrightarrow{L}) & \text{if }k\text{ is even}\\
0 & \text{if }k\text{ is odd}%
\end{array}
\right.
\]

for $k=1,\ldots,q-1$ and
\[
b_{q}=\left \{
\begin{array}
[c]{ll}%
-\underset{\overrightarrow{C}\in \overrightarrow{\mathcal{C}}_{q,\omega
^{\prime}}^{+}}{\sum}\omega(\overrightarrow{C})+\underset{\overrightarrow
{C}\in \overrightarrow{\mathcal{C}}_{q,\omega^{\prime}}^{-}}{\sum}%
\omega(\overrightarrow{C})+\underset{\overrightarrow{L}\in \overrightarrow
{\mathcal{D}}_{q}}{\sum}\omega(\overrightarrow{L}) & \text{if }q\text{ is
even}\\
-\underset{\overrightarrow{C}\in \overrightarrow{\mathcal{C}}_{q,\omega
^{\prime}}^{+}}{\sum}\omega(\overrightarrow{C})+\underset{\overrightarrow
{C}\in \overrightarrow{\mathcal{C}}_{q,\omega^{\prime}}^{-}}{\sum}%
\omega(\overrightarrow{C}) & \text{if }q\text{ is odd}%
\end{array}
\right.
\]

\end{cor}

\begin{proof}
The first equality follows from Lemma \ref{coeff a balanced}.

From formula (\ref{2}), we have%
\[%
\begin{array}
[c]{cl}%
b_{q} & =\underset{\overrightarrow{L}\in \overrightarrow{\mathcal{L}}_{q}}%
{\sum}(-1)^{\left \vert \overrightarrow{L}\right \vert }\omega^{\prime
}(\overrightarrow{L})\\
& =\underset{\overrightarrow{L}\in \overrightarrow{\mathcal{L}}_{q}%
\setminus \overrightarrow{\mathcal{L}}_{q}^{e}}{\sum}(-1)^{\left \vert
\overrightarrow{L}\right \vert }\omega^{\prime}(\overrightarrow{L}%
)+\underset{\overrightarrow{L}\in \overrightarrow{\mathcal{D}}_{q}}{\sum
}(-1)^{\left \vert \overrightarrow{L}\right \vert }\omega^{\prime}%
(\overrightarrow{L})-\underset{\overrightarrow{C}\in \overrightarrow
{\mathcal{C}}_{q}}{\sum}\omega^{\prime}(\overrightarrow{C})
\end{array}
\]

By definition of $\overrightarrow{\mathcal{C}}_{q,\omega^{\prime}}^{+}$ and
$\overrightarrow{\mathcal{C}}_{q,\omega^{\prime}}^{-}$, we have $\underset
{\overrightarrow{C}\in \overrightarrow{\mathcal{C}}_{q}}{\sum}\omega^{\prime
}(\overrightarrow{C})=\underset{\overrightarrow{C}\in \overrightarrow
{\mathcal{C}}_{q,\omega^{\prime}}^{+}}{\sum}\omega(\overrightarrow
{C})-\underset{\overrightarrow{C}\in \overrightarrow{\mathcal{C}}%
_{q,\omega^{\prime}}^{-}}{\sum}\omega(\overrightarrow{C})$.

Consider now $\overrightarrow{L}\in \overrightarrow{\mathcal{L}}_{q}%
\setminus \overrightarrow{\mathcal{L}}_{q}^{e}$. By definition of
$\overrightarrow{\mathcal{L}}_{q}$ and $\overrightarrow{\mathcal{L}}_{q}^{e}$,
the linear subdigraph $\overrightarrow{L}$ contains an odd cycle
$\overrightarrow{C}$ among its components. Let $\overrightarrow{L}^{\prime}$
the linear subdigraph obtained from $\overrightarrow{L}$ by replacing the
cycle $\overrightarrow{C}$ by $\overrightarrow{C}^{\ast}$. Since
$\overrightarrow{C}$ is odd and $\omega(\overrightarrow{C})=\omega
(\overrightarrow{C}^{\ast})$, $\omega^{\prime}(\overrightarrow{L}%
)=-\omega^{\prime}(\overrightarrow{L}^{\prime})$. Thus, linear subdigraphs of
$\overrightarrow{\mathcal{L}}_{q}\setminus \overrightarrow{\mathcal{L}}_{q}%
^{e}$ contribute $0$ to $b_{p}$. Now, according to the parity of $p$, we
have$\underset{\overrightarrow{L}\in \overrightarrow{\mathcal{D}}_{p}}{\sum
}(-1)^{\left \vert \overrightarrow{L}\right \vert }\omega^{\prime}%
(\overrightarrow{L})=\underset{\overrightarrow{L}\in \overrightarrow
{\mathcal{D}}_{p}}{\sum}(-1)^{p/2}(-1)^{p/2}\omega(\overrightarrow
{L})=\underset{\overrightarrow{L}\in \overrightarrow{\mathcal{D}}_{p}}{\sum
}\omega(\overrightarrow{L})$\ if\ $p$ is even and $0$ if\ $p$ is odd, which
yields the second equality in the Corollary.
\end{proof}

It follows that if $(D,\omega)$ is cycle-symmetric and contains no even cycles of
length greater than $2$ then the characteristic polynomial of $(D,\omega
^{\prime})$ is the same for any skew-signing $\omega^{\prime}$ of $(D,\omega
)$. This proves the implication $ii)\Rightarrow i)$ of Proposition
\ref{skew balanced}. The proof of $i)$ implies $ii)$ is a direct consequence
of the following more general result.\

\begin{lem}
\label{odd dicycle} Let $(D,\omega)$ be a $(\leq l)$-cycle-symmetric pwsl-digraph
where $l\geqslant3$. If the characteristic polynomial of $(D,\omega^{\prime})$
is the same for any skew-signing $\omega^{\prime}$ of $(D,\omega)$, then every
cycle of length at most $l$ is an odd cycle or a digon.
\end{lem}

Before proving this Lemma, we introduce some notations and establish an
intermediate result. Let $(D,\omega)$ be a an arbitrary pwsl-digraph and
consider an arbitrary cycle of $D$ of length $q\geqslant3$ whose vertices
are $v_{1},\ldots,v_{q}$ and whose arcs are $e_{1}:=v_{1}v_{2},\ldots
,e_{q-1}:=v_{q-1}v_{q}$, $e_{q}:=v_{q}v_{1}$. Let $\omega^{\prime}$ be a skew-
signing of\textbf{ }$(D,\omega)$\textbf{. }For $h\in \{1,...,q\}$ and
$r\in \{1,...,h\}$, we denote by $\eta_{\omega^{\prime}}^{+}(e_{1},\ldots
,e_{r},\overline{e_{r+1}},\ldots,\overline{e_{h}})$ the sum of the weights of
cycles $\overrightarrow{C}$ of length $q$ in $D$ that have $\omega^{\prime
}(\overrightarrow{C})>0$ and contain arcs $e_{1},\ldots,e_{r}$ but not arcs
$e_{r+1},\ldots,e_{h}$,. Define $\eta_{\omega^{\prime}}^{-}(e_{1},\ldots
,e_{r},\overline{e_{r+1}},\ldots,\overline{e_{h}})$ analogously.

\begin{lem}
\label{intermid}There exists a skew-signing $\omega_{0}^{\prime}$ of
$(D,\omega)$ such that $\eta_{\omega_{0}^{\prime}}^{+}(e_{1})\neq \eta
_{\omega_{0}^{\prime}}^{-}(e_{1})$.
\end{lem}

\begin{proof}
Assume the contrary. We claim that for each $t\in \left \{  1,\ldots,q\right \} $ and for all skew-signing $\omega^{\prime}$ of
$(D,\omega)$, $\eta_{\omega^{\prime}}^{+}(e_{1},\ldots,e_{t})=\eta_{\omega^{\prime}}%
^{-}(e_{1},\ldots,e_{t})$. For this, we proceed by induction on $t$. The case $t=1$ is
assumed. Let $t\in \left \{  1,\ldots,q-1\right \}  $ and suppose that the claim
is true for $t$. Then%
\[
\left \{
\begin{array}
[c]{c}%
\eta_{\omega^{\prime}}^{+}(e_{1},\ldots,e_{t})=\eta_{\omega^{\prime}}%
^{+}(e_{1},e_{2},\ldots,e_{t},e_{t+1})+\eta_{\omega^{\prime}}^{+}(e_{1}%
,e_{2},\ldots,e_{t},\overline{e_{t+1}})\\
\eta_{\omega^{\prime}}^{-}(e_{1},\ldots,e_{t})=\eta_{\omega^{\prime}}%
^{-}(e_{1},e_{2},\ldots,e_{t},e_{t+1})+\eta_{\omega^{\prime}}^{-}(e_{1}%
,e_{2},\ldots,e_{t},\overline{e_{t+1}})
\end{array}
\right.
\]

Consider now the skew-signing $\omega^{\prime \prime}$ that coincides with
$\omega^{\prime}$ outside $\left \{  e_{t+1},e_{t+1}^{\ast}\right \}  $ and such
that $\omega^{\prime \prime}(e)=-\omega^{\prime}(e)$\ for $e\in \left \{
e_{t+1},e_{t+1}^{\ast}\right \}  $.

Then, we have
\[
\left \{
\begin{array}
[c]{c}%
\eta_{\omega^{\prime \prime}}^{+}(e_{1},\ldots,e_{t})=\eta_{\omega^{\prime}%
}^{-}(e_{1},e_{2},\ldots,e_{t},e_{t+1})+\eta_{\omega^{\prime}}^{+}(e_{1}%
,e_{2},\ldots,e_{t},\overline{e_{t+1}})\\
\eta_{\omega^{\prime \prime}}^{-}(e_{1},\ldots,e_{t})=\eta_{\omega^{\prime}%
}^{+}(e_{1},e_{2},\ldots,e_{t},e_{t+1})+\eta_{\omega^{\prime}}^{-}(e_{1}%
,e_{2},\ldots,e_{t},\overline{e_{t+1}})
\end{array}
\right.
\]

But by induction hypothesis, we have  $\eta_{\omega^{\prime \prime}}^{+}(e_{1},\ldots,e_{t})=\eta_{\omega
^{\prime \prime}}^{-}(e_{1},\ldots,e_{t})$ and $\eta_{\omega^{\prime}}%
^{+}(e_{1},\ldots,e_{t})=\eta_{\omega^{\prime}}^{-}(e_{1},\ldots,e_{t})$.

Then 

\[%
\begin{array}
[c]{cl}%
\eta_{\omega^{\prime}}^{+}(e_{1},\ldots,e_{t},\overline{e_{t+1}})-\eta
_{\omega^{\prime}}^{-}(e_{1},\ldots,e_{t},\overline{e_{t+1}}) &
=\eta_{\omega^{\prime}}^{-}(e_{1},\ldots,e_{t+1})-\eta_{\omega^{\prime}%
}^{+}(e_{1},\ldots,e_{t+1})\\
& =\eta_{\omega^{\prime}}^{+}(e_{1},\ldots,e_{t+1})-\eta_{\omega
^{\prime}}^{-}(e_{1},\ldots,e_{t+1})
\end{array}
\]

Thus $\eta_{\omega^{\prime}}^{+}(e_{1},\ldots,e_{t},e_{t+1})=\eta
_{\omega^{\prime}}^{-}(e_{1},e_{2},\ldots,e_{t},e_{t+1})$.

This complete the induction proof. For $t=q$ we have , $\eta_{\omega^{\prime}}%
^{+}(e_{1},\ldots,e_{q})=\eta_{\omega^{\prime}}^{-}(e_{1},\ldots,e_{q})$.

 Now, choose a skew-signing $\omega^{\prime}$ of $\left(  D,\omega \right)  $ such
that $\omega^{\prime}(e_{1})=\omega(e_{1}),\ldots,\omega^{\prime}%
(e_{q})=\omega(e_{q})$. Then, we have $\eta_{\omega^{\prime}}^{+}(e_{1}%
,\ldots,e_{q})=\overset{q}{\underset{i=1}{\prod}}\omega(e_{i})$ and$\  \eta
_{\omega^{\prime}}^{-}(e_{1},\ldots,e_{q})=0$, a contradiction. It follows
that there exists a skew-signing $\omega_{0}^{\prime}$ such that $\eta
_{\omega_{0}^{\prime}}^{+}(e_{1})\neq \eta_{\omega_{0}^{\prime}}^{-}(e_{1})$.
\end{proof}

\begin{pol}
 Assume for contradiction that $D$ contains an
even cycle of length $q\in \left \{  4,\ldots,l\right \}  $ and choose such a
cycle with $q$ as small as possible. We will use the notations of the
previous lemma. Let $\omega^{\prime \prime}$ be the skew-signing of
$(D,\omega)$ that coincides with $\omega_{0}^{\prime}$ outside $\left \{
e_{1},e_{1}^{\ast}\right \}  $ and such that $\omega^{\prime \prime}%
(e)=-\omega_{0}^{\prime}(e)$\ for $e\in \left \{  e_{1},e_{1}^{\ast}\right \}  $.
The charateristic\ polynomials of $(D,\omega_{0}^{\prime})$ and $(D,\omega
^{\prime \prime})$ are respectively denoted by $p_{(D,\omega_{0}^{\prime}%
)}(x):=x^{n}+b_{1}x^{n-1}+\cdots+b_{n-1}x+b_{n}$ and $p_{(D,\omega
^{\prime \prime})}(x):=x^{n}+c_{1}x^{n-1}+\cdots+c_{n-1}x+c_{n}$.

By the choice of $q$ and from the second equality of Corollary \ref{coeff}, we
have $b_{q}-c_{q}=-\underset{\overrightarrow{C}\in \overrightarrow{\mathcal{C}%
}_{q,\omega_{0}^{\prime}}^{+}}{\sum}\omega(\overrightarrow{C})+\underset
{\overrightarrow{C}\in \overrightarrow{\mathcal{C}}_{q,\omega_{0}^{\prime}}%
^{-}}{\sum}\omega(\overrightarrow{C})+\underset{\overrightarrow{C}%
\in \overrightarrow{\mathcal{C}}_{q,\omega"}^{+}}{\sum}\omega(\overrightarrow
{C})-\underset{\overrightarrow{C}\in \overrightarrow{\mathcal{C}}_{q,\omega
"}^{-}}{\sum}\omega(\overrightarrow{C})$.

Every cycle $\overrightarrow{C}$ of length $q$ that contains neither $e_{1}$
nor $e_{1}^{\ast}$ contributes $0$ to $b_{q}-c_{q}$. It follows that:

\[%
\begin{array}
[c]{cl}%
b_{q}-c_{q} & =-\eta_{\omega_{0}^{\prime}}^{+}(e_{1})-\eta_{\omega_{0}%
^{\prime}}^{+}(e_{1}^{\ast})+\eta_{\omega_{0}^{\prime}}^{-}(e_{1}%
)+\eta_{\omega_{0}^{\prime}}^{-}(e_{1}^{\ast})\\
& +\eta_{\omega^{\prime \prime}}^{+}(e_{1})+\eta_{\omega^{\prime \prime}}%
^{+}(e_{1}^{\ast})-\eta_{\omega^{\prime \prime}}^{-}(e_{1})-\eta_{\omega
^{\prime \prime}}^{-}(e_{1}^{\ast})
\end{array}
\]

By construction of $\omega^{\prime \prime}$, we have $\eta_{\omega
^{\prime \prime}}^{+}(e_{1})=\eta_{\omega_{0}^{\prime}}^{-}(e_{1})$,
$\eta_{\omega^{\prime \prime}}^{-}(e_{1})=\eta_{\omega_{0}^{\prime}}^{+}%
(e_{1})$, $\eta_{\omega^{\prime \prime}}^{+}(e_{1}^{\ast})=\eta_{\omega
_{0}^{\prime}}^{-}(e_{1}^{\ast})$, $\eta_{\omega^{\prime \prime}}^{-}%
(e_{1}^{\ast})=\eta_{\omega_{0}^{\prime}}^{+}(e_{1}^{\ast})$.

Then $b_{q}-c_{q}=-2(\eta_{\omega_{0}^{\prime}}^{+}(e_{1})+\eta_{\omega
_{0}^{\prime}}^{+}(e_{1}^{\ast}))+2(\eta_{\omega_{0}^{\prime}}^{-}(e_{1}%
)+\eta_{\omega_{0}^{\prime}}^{-}(e_{1}^{\ast}))$

As $(D,\omega)$ is $(\leq l)$-cycle-symmetric, we have $\eta_{\omega_{0}^{\prime}%
}^{+}(e_{1})=\eta_{\omega_{0}^{\prime}}^{+}(e_{1}^{\ast})$, $\eta_{\omega
_{0}^{\prime}}^{-}(e_{1}^{\ast})=\eta_{\omega_{0}^{\prime}}^{-}(e_{1})$ and
then $b_{q}-c_{q}=-4(\eta_{\omega_{0}^{\prime}}^{+}(e_{1})-\eta_{\omega
_{0}^{\prime}}^{-}(e_{1}))\neq0$, a contradiction. 
\end{pol}

\section{Proof of the main theorem}

The implication $ii)\Longrightarrow i)$ follows easily from Corollary
\ref{balanced diagonally simillar} and Proposition  \ref{skew balanced}. To prove
$i)\Longrightarrow ii)$ it suffices to use Proposition \ref{skew balanced} and
the next Lemma.

\begin{lem}
Let $(D,\omega)$ be a pwls-digraph. If the characteristic polynomial of
$(D,\omega^{\prime})$ is the same for any skew-signing $\omega^{\prime}$ of
$(D,\omega)$, then $(D,\omega)$ is cycle-symmetric.
\end{lem}

\begin{proof}

Assume for contradiction that $(D,\omega)$ is not cycle-symmetric and let
$\overrightarrow{C}_{0}$ be a shortest cycle of $D$ such that $\omega
(\overrightarrow{C}_{0})\neq \omega(\overrightarrow{C}_{0}^{\ast})$. We denote
by $v_{1},\ldots,v_{q}$ the vertices of $\overrightarrow{C}_{0}$ and
$e_{1}:=v_{1}v_{2},\ldots,e_{q-1}:=v_{q-1}v_{q}$, $e_{q}:=v_{q}v_{1}$ its
arcs. Let $\omega^{\prime}$ be a skew-signing of\textbf{ }$(D,\omega
)$\textbf{. }

For $h\in \{1,...,q\}$ and $r\in \{1,...,h\}$, we set%

\begin{tabular}
[c]{ll}%
$N_{\omega^{\prime}}^{+}(e_{1},\ldots,e_{r},\overline{e_{r+1}},\ldots
,\overline{e_{h}})=$ & $\eta_{\omega^{\prime}}^{+}(e_{1},\ldots,e_{r}%
,\overline{e_{r+1}},\ldots,\overline{e_{h}})$\\
& $+\eta_{\omega^{\prime}}^{+}(e_{1}^{\ast},\ldots,e_{r}^{\ast},\overline
{e_{r+1}^{\ast}},\ldots,\overline{e_{h}^{\ast}})$\\
$N_{\omega^{\prime}}^{-}(e_{1},\ldots,e_{r},\overline{e_{r+1}},\ldots
,\overline{e_{h}})=$ & $\eta_{\omega^{\prime}}^{-}(e_{1},\ldots,e_{r}%
,\overline{e_{r+1}},\ldots,\overline{e_{h}})$\\
& $+\eta_{\omega^{\prime}}^{-}(e_{1}^{\ast},\ldots,e_{r}^{\ast},\overline
{e_{r+1}^{\ast}},\ldots,\overline{e_{h}^{\ast}})$%
\end{tabular}

\textbf{Step 1} There exists a skew-signing $\omega_{0}^{\prime}$ of
$(D,\omega)$ such that $N_{\omega_{0}^{\prime}}^{+}(e_{1})\neq N_{\omega
_{0}^{\prime}}^{-}(e_{1})$.

Assume by contradiction that $N_{\omega^{\prime}}^{+}(e_{1})=N_{\omega
^{\prime}}^{-}(e_{1})$ for every skew-signing $\omega^{\prime}$ of
$(D,\omega)$. By using an induction process, we can deduce, as in the proof of
Lemma \ref{intermid}, that $N_{\omega^{\prime}}^{+}(e_{1},\ldots
,e_{q})=N_{\omega^{\prime}}^{-}(e_{1},\ldots,e_{q})$. However,
\[
N_{\omega^{\prime}}^{+}(e_{1},\ldots,e_{q})=\left \{
\begin{array}
[c]{l}%
\omega(\overrightarrow{C}_{0})+\omega(\overrightarrow{C}_{0}^{\ast})\text{
\  \  \  \  \  \ if }q\text{ is even and }\omega^{\prime}(\overrightarrow{C}%
_{0})>0\text{ }\\
0\text{ \  \  \  \  \  \  \  \  \  \  \  \  \  \  \  \  \  \  \  \  \  \  \  \  \ if }q\text{ is even
and }\omega^{\prime}(\overrightarrow{C}_{0})<0\\
\omega(\overrightarrow{C}_{0})\text{\  \  \  \  \  \  \  \  \  \  \  \  \  \  \  \  \  \  \ if
}q\text{ is odd and }\omega^{\prime}(\overrightarrow{C}_{0})>0\\
\omega(\overrightarrow{C}_{0}^{\ast}%
)\text{\  \  \  \  \  \  \  \  \  \  \  \  \  \  \  \  \  \  \ if }q\text{ is odd and }%
\omega^{\prime}(\overrightarrow{C}_{0})<0
\end{array}
\right.
\]

and%

\[
N_{\omega^{\prime}}^{-}(e_{1},\ldots,e_{q})=\left \{
\begin{array}
[c]{l}%
0\text{ \  \  \  \  \  \  \  \  \  \  \  \  \  \  \  \  \  \  \  \  \  \  \  \  \  \ if }q\text{ is
even and }\omega^{\prime}(\overrightarrow{C}_{0})>0\text{ }\\
\omega(\overrightarrow{C}_{0})+\omega(\overrightarrow{C}_{0}^{\ast
})\text{\  \  \  \  \  \  \  \ if }q\text{ is even and }\omega^{\prime}%
(\overrightarrow{C}_{0})<0\\
\omega(\overrightarrow{C}_{0}^{\ast})\text{
\  \  \  \  \  \  \  \  \  \  \  \  \  \  \  \  \  \  \ if }q\text{ is odd and }\omega^{\prime
}(\overrightarrow{C}_{0})>0\\
\omega(\overrightarrow{C}_{0})\text{ \  \  \  \  \  \  \  \  \  \  \  \  \  \  \  \  \  \  \ if
}q\text{ is odd and }\omega^{\prime}(\overrightarrow{C}_{0})<0
\end{array}
\right.
\]

which contradicts our assumption on $\overrightarrow{C}_{0}$. This complete
the proof of Step 1.

\textbf{Step 2.} $(D,\omega)$ is $(\leq q-1)$-cycle-symmetric and contains no
even cycles of length $k\in$ $\left \{  3,\ldots,q-1\right \}  $.

This follows from the choice of $q$ and lemma \ref{odd dicycle}.

Consider now the skew-signing $\omega^{\prime \prime}$ of $(D,\omega)$ that
coincides with $\omega^{\prime}$ outside $\left \{  e_{1},e_{1}^{\ast}\right \}
$ and such that $\omega^{\prime \prime}(e)=-\omega_{0}^{\prime}(e)$\ for
$e\in \left \{  e_{1},e_{1}^{\ast}\right \}  $. Let $p_{(D,\omega_{0}^{\prime}%
)}(x):=x^{n}+b_{1}x^{n-1}+\cdots+b_{n-1}x+b_{n}$ and $p_{(D,\omega
^{\prime \prime})}(x):=x^{n}+c_{1}x^{n-1}+\cdots+c_{n-1}x+c_{n}$ be the
charateristic\ polynomials of $(D,\omega_{0}^{\prime})$ and $(D,\omega
^{\prime \prime})$ respectively.

As in the proof of Lemma \ref{odd dicycle}, we have%
\[%
\begin{array}
[c]{cl}%
b_{q}-c_{q} & =-2(\eta_{\omega_{0}^{\prime}}^{+}(e_{1})+\eta_{\omega
_{0}^{\prime}}^{+}(e_{1}^{\ast}))+2(\eta_{\omega_{0}^{\prime}}^{-}(e_{1}%
)+\eta_{\omega_{0}^{\prime}}^{-}(e_{1}^{\ast}))\\
& =-2(N_{\omega_{0}^{\prime}}^{+}(e_{1})-N_{\omega_{0}^{\prime}}^{-}%
(e_{1}))\neq0
\end{array}
\]
which contradicts Step 1. This ends the proof of Lemma.
\end{proof}

\end{document}